\documentclass{siamart190516}
\usepackage{times}
\usepackage[T1]{fontenc}

\usepackage{multirow}
\usepackage{braket,amsfonts,amsopn,bm}
\usepackage{graphicx,epstopdf}
\usepackage{array}
\usepackage[linesnumbered,ruled,vlined,algo2e]{algorithm2e}
\usepackage{mathrsfs}

\usepackage{amssymb,amsfonts,amsmath,latexsym,epsfig,euscript,color}

\newcommand{\RR}{\mathbb R}

\newcommand{\NN}{\mathbb N}
\usepackage{bm}

\newcommand{\bmm}{{\bm m}}
\newcommand{\bn}{{\bm n}}
\newcommand{\bW}{{\bm W}}
\newcommand{\bX}{{\bm X}}
\newcommand{\bY}{{\bm Y}}
\newcommand{\bH}{{\bm H}}
\newcommand{\bS}{{\bm S}}
\newcommand{\bM}{{\bm M}}
\newcommand{\bK}{{\bm K}}
\newcommand{\bN}{{\bm N}}
\newcommand{\bR}{{\bm R}}
\newcommand{\bC}{{\bm C}}
\newcommand{\bD}{{\bm D}}
\newcommand{\bQ}{{\bm Q}}
\newcommand{\bF}{{\bm F}}
\newcommand{\bZ}{{\bm Z}}
\newcommand{\bG}{{\bm G}}
\newcommand{\bI}{{\bm I}}
\newcommand{\bA}{{\bm A}}
\newcommand{\bB}{{\bm B}}
\newcommand{\bLambda}{{\bm \Lambda}}
\newcommand{\bE}{{\bm E}}
\newcommand{\bd}{{\bm d}}
\newcommand{\bu}{{\bm u}}

\newcommand{\bv}{{\bm v}}
\newcommand{\blambda}{{\pmb \lambda}}
\newcommand{\bbf}{{\bm f}}

\newcommand{\trace}{{\rm trace}}
\newtheorem{remark}[theorem]{Remark}
\newtheorem{example}[theorem]{Example}

\newcommand{\XX}{{\bm X}}

\newcommand{\bx}{{\bm x}}

\begin{document}

\title{Numerical solution of a class of quasi-linear matrix equations
\thanks{Version of August 11, 2022. The authors are members of the INdAM Research Group GNCS that partially supported this  work.} }{}
\date{\today}
\author{Margherita Porcelli\thanks{Dipartimento di Matematica, AM$^2$,
        Alma Mater Studiorum - Universit\`a di Bologna, Piazza di Porta San Donato 5,
        40126 Bologna, Italia. Emails:
{\tt  \{margherita.porcelli,valeria.simoncini\}@unibo.it}}\and{\,}\thanks{ISTI--CNR, Via Moruzzi 1, Pisa, Italia}
 \and Valeria Simoncini$^{\dagger,}$\thanks{IMATI-CNR, Via Ferrata 5/A, Pavia, Italia}}

\maketitle

\begin{abstract}
Given the matrix equation
$\bA \XX +  \XX \bB + f(\XX) \bC =\bD$ in the unknown $n\times m$ matrix $\bX$,
we analyze existence and uniqueness conditions, together with
computational solution strategies for
$f \,: \RR^{n \times m} \to \RR$ being a linear or nonlinear function.
We characterize different properties of the matrix equation and of its
solution, depending on the considered classes of functions $f$.
Our analysis mainly concerns small dimensional problems, though several
considerations also apply to large scale matrix equations.
\end{abstract}



\section{The problem}
We consider the following nonlinear equation
\begin{equation}\label{eqn:main}
\bA \XX +  \XX \bB + f(\XX) \bC =\bD ,
\end{equation}
in the unknown matrix $\bX\in\RR^{n\times m}$,
where $f : \RR^{n\times m} \to \RR$ is a linear or nonlinear function, 
while $\bA\in\RR^{n\times n}$, $\bB\in\RR^{m\times m}$, and $\bC, \bD\in\RR^{n\times m}$ are 
given matrices. Throughout the
paper we assume that $\bA$ and $-\bB$ have no common eigenvalues, so that
the operator ${\cal L} : \XX \mapsto \bA\XX+\XX \bB$ is invertible.
Since the nonlinear function in $\bX$ yields a scalar contribution to the matrix
equation, we will refer to this problem as a {\em quasi-linear} matrix equation.
We also notice that depending on the type of function $f$,
the condition $m=n$ may also hold, and this will be assumed throughout without explicit mention.

 Equation (\ref{eqn:main}) is among the simplest possible 
generalizations of the
Sylvester equation to more than two terms in the unknown matrix $\bX$.
Yet, it provides different intriguing challenges for its numerical solution,
that we aim to address. A natural further generalization is the inclusion
of more quasi-linear terms. 
Our interest in this problem stems from certain applications with linear $f$, see section~\ref{sec:appl} and
\cite{Padovani.Porcelli.22}, however we believe that the general case of $f$ nonlinear
may find applications in different contexts where the given mathematical
problem can be formulated in terms of a matrix equation. To the best of our knowledge,
no numerical methods have been presented in the literature
for the class of problems considered in (\ref{eqn:main}).

To begin our analysis, we observe that by letting
$\bN=-{\cal L}^{-1}(\bC)$, $\bM={\cal L}^{-1}(\bD)$, problem (\ref{eqn:main}) is
mathematically equivalent to
\begin{equation}\label{eqn:main1}
 \XX = \bM+  f(\XX) \bN.
\end{equation}
This equation provides the ideal computational setting in case $n, m$ are small,
as attention can be put into the function $f$, assuming that
$\bM, \bN$ can be computed accurately. In case of matrices with large
dimensions, the equality in (\ref{eqn:main1}) will in general be replaced by an approximation.

We start by considering the case of linear $f$, which was motivated by an
application in solid mechanics and civil engineering developed in \cite{Padovani.Porcelli.22}.
A linear function $f : \RR^{n\times m} \to \RR$ can be defined  as 
$f(\bX) = {\rm trace}(\bH \bX)$ for some matrix $\bH$ of appropriate dimensions.
For instance, for $\bH$ equal to the identity matrix and $\bX$ square, $f(\bX) = {\rm trace}(\bX)$,
while for $\bH=\bu\bv^T$ with $\bu\in\RR^m, \bv\in\RR^n$, $f(\bX) = \bv^T \bX \bu$, where the properties of the
trace have been used.
We will derive a closed form for $\bX$, and also observe that under
certain hypotheses $f(\bX)$ may be obtained without explicitly computing $\bX$.

We then analyze a more general setting where
$f$ is the composition of a linear and a nonlinear function. The order in
which these two functions are combined significantly influences the analysis and results:
as an example, different existence and uniqueness properties may hold. So,
for instance, working with
$f(\bX)  = {\rm trace}(\exp(-\bX))$ (linear combined with nonlinear) differs significantly from dealing with
$f(\bX)  = \exp(-{\rm trace}(\bX))$ (nonlinear combined with linear). 
Distinct computational procedures also need to be devised. 

We will explore iterative techniques that appropriately handle both $f$
and the matrices forming the linear part of the equation.
The linear-nonlinear problem is more computationally involved as the
iteration requires matrix function evaluations and matrix updates. For this
problem we will 
derive convergence results for a natural fixed-point iteration. In
the nonlinear-linear case, the nonlinear iteration is performed at the scalar level
and classical results for nonlinear equations can be employed, while taking into account
the properties of the given data.

The following notation is adopted: matrices (resp. vectors) are denoted by bold case capital (resp. small) 
roman letters, while small roman letters are used for real valued functions (with the exception
of matrix indices and matrix dimensions), and greek letters are used for scalars.
The notation $\bA \succ 0$ ($\bA \succeq 0$) denotes
a symmetric and positive definite (semidefinite) matrix $\bA$; the notation $\bA \succeq \bB$ is equivalent
to $\bA -\bB \succeq 0$. For a given matrix $\bX$, the operator {\rm vec}($\bX$) stacks all columns
of $\bX$ one below the other into a single long vector, while the Kronecker operator of
two matrices $\bA\in \RR^{n_A\times m_A}$ and  $\bB\in \RR^{n_B\times m_B}$, is given by
$$
  \bA \otimes \bB = 
 \begin{bmatrix}
  {a}_{11}B & {a}_{12} \bB & \cdots & {a}_{1 m_A} \bB \\
  {a}_{21}B & {a}_{22} \bB & \cdots & {a}_{2 m_A} \bB \\
  \vdots &          &        & \vdots      \\
  {a}_{n_A 1}\bB & {a}_{n_A 2} \bB & \cdots & {a}_{n_A m_A} \bB \\
 \end{bmatrix} \in \RR^{n_A n_B \times m_A m_B}.
$$

\section{The case of linear $f$}
The following proposition yields the solution of (\ref{eqn:main}) in closed form
when $f$ is a linear function.

\begin{proposition}\label{prop:f(X)}  
Let $\bM, \bN$ be the solutions to the 
Sylvester equations $\bA\bM+\bM\bB=\bD$ and $\bA\bN+\bN\bB=-\bC$, respectively. Assume that
$1-f(\bN)\ne 0$. Then the solution to (\ref{eqn:main}) is given by
$$
 \bX = \bM + \sigma \bN, \quad \sigma = \frac{f(\bM)}{1-f(\bN)}.
$$
\end{proposition}

\begin{proof}
The problem can be written as in (\ref{eqn:main1}).
Applying the linear function $f$ to both sides of (\ref{eqn:main1}) yields
$f(\bX) = f(\bM) + f(\bX) f(\bN)$, that is
$f(\bX) = f(\bM)/(1-  f(\bN))$.
Substituting in (\ref{eqn:main1}) the expression for $\bX$ follows. Finally, by linearity,
\begin{eqnarray*}
\bA \XX +  \XX \bB + f(\XX) \bC -\bD  &=& 
\bA (\bM + \sigma \bN) + (\bM + \sigma \bN) \bB + f(\XX) \bC -\bD \\
&=&  \sigma \bA \bN + \sigma \bN \bB + f(\XX) \bC  = -\sigma \bC + f(\XX) \bC=0,
\end{eqnarray*}
which verifies that $\bX$ solves (\ref{eqn:main}).
\end{proof}

In case it holds that $1-  f(\bN)=0$, the relation $f(\bX) = f(\bM) + f(\bX) f(\bN)$ shows
that two possible scenarios arise: for $f(\bM)=0$ then $\bX = \bM + \sigma \bN$ where $\sigma$
can be any real number, yielding a nonunique solution; if $f(\bM)\ne 0$ then no solutions exist.

It is interesting to observe that when $f(\bX)= \bu^T \bX \bu$ and $\bC= \bv \bv^T$,
it holds that $f(\bX) C = (\bv \bu^T) \bX \bu \bv^T$. The problem thus corresponds to
the linear matrix equation
$$
\bA \bX + \bX \bB + \bK^T \bX \bK = \bD
$$
with $\bK=\bu\bv^T$ of rank one; more generally, depending on $f$ and $\bC$, the term $\bK^T \bX \bK$ may
take the form $\bK_1 \bX \bK_2$, with $\bK_2$ not necessarily the transpose of $\bK_1$.
With the previous choice of $\bK$, the closed form solution  in Proposition~\ref{prop:f(X)} is equivalent
to the Sherman-Morrison-Woodbudy (SMW) formula obtained by the vector form of the matrix equation.
Indeed, the matrix equation above can be written as $(\bG + {\cal V}{\cal U}^T) \bx = \bd$,
where $\bG = \bI\otimes \bA +  \bB^T \otimes \bI$, ${\cal U}=\bu\otimes\bu$,
${\cal V} = \bv\otimes\bv$ and $\bd = {\rm vec}(\bD)$; see, e.g., \cite{Hao.Simoncini.21b}.
Then the SMW formula reads
\begin{eqnarray*}
\bx &=& \bG^{-1} \bd - \bG^{-1}{\cal V} (1 + {\cal U}^T \bG^{-1} {\cal V})^{-1} {\cal U}^T
\bG^{-1} \bd \\
&=& {\rm vec}( {\cal L}^{-1}(\bD)) - \sigma \, {\rm vec}( {\cal L}^{-1}(\bC)),
\end{eqnarray*}
where $\sigma = (1 + {\cal U}^T \bG^{-1} {\cal V})^{-1} {\cal U}^T
\bG^{-1} \bd$, which precisely corresponds to $\sigma$ in Proposition~\ref{prop:f(X)},
as ${\cal U}^T \bG^{-1} {\cal V}=\bu^T {\cal L}^{-1}(\bv \bv^T) \bu = {-f(\bN)}$ and
similarly for the other quantities. The cases where ${\cal U}=\bu_1\otimes\bu_2$,
${\cal V} = \bv_1\otimes\bv_2$ can be treated analogously.

It is also noticeable that for $f(\bX)=\trace(\bX)$ and $\bB=\bA$ with
$\bA$ nonsingular, the quantity $\trace(\bX)$ can be obtained
without solving two Sylvester equations, but by only solving linear systems with $\bA$.
Indeed, from $\bA \bX +  \bX \bA + \trace(\bX) \bC = \bD$
we write
$$
 \bX +  \bA^{-1} \bX \bA + \trace(\bX)\bA^{-1}  \bC =\bA^{-1}  \bD.
$$
Applying the trace to all matrix terms we obtain
$$
\trace(\bX) = \frac{\trace(\bA^{-1}  \bD)}{2 + \trace(\bA^{-1}  \bC)},
$$
where we have used the linearity and cyclic property of the trace.
After this computation, the final $\bX$ is obtained by solving the Sylvester equation
$\bA \bX +  \bX \bA = \bD- \trace(\bX) \bC$.
The actual number of systems with $\bA$ depends on the structure of $\bD$ and $\bC$.
For instance, if $\bC=\bC_1\bC_2^T$ has low rank equal to $k$ and $\bC_1\in\RR^{n\times k}$, 
then only $k$ systems with $\bA$ need to be solved to compute
$\trace(\bA^{-1}  \bC)= \trace(\bC_2^T \bA^{-1}  \bC_1)$. Other properties 
of the involved matrices can be exploited to lower the computational efforts.


\begin{remark}\label{rem:solve}
The trace of the Sylvester solution matrix is of interest in its own right; see, e.g., 
\cite{Wang.Kuo.Hsu.86},\cite{Truhar.Veselic.07},\cite{Savov.Popchev.08} and their references.
In particular, for $\bB=\bA$ symmetric and $\bC=0$,
the procedure discussed in Proposition~\ref{prop:f(X)} can be used to compute the trace of 
the solution to $\bA \bX +  \bX \bA  = \bD$, without
explicitly computing or approximating the solution matrix.

This fact can be used for instance if in problem (\ref{eqn:main}) one is interested in
only computing the trace of $\bX$, and not $\bX$ itsself.  In this case, 
trace($\bM)$, trace($\bN)$ can be obtained without explicitly computing the two matrices $\bM, \bN$.
\end{remark}

In a way similar to Proposition~~\ref{prop:f(X)} one can treat the related problem
\begin{equation}\label{eqn:main_multi}
\bA \bX +  \bX \bB + f_1(\bX) \bC_1 + \ldots f_\ell(\bX) \bC_\ell =\bD ,
\end{equation}
where $f_i$, $i=1, \ldots, \ell$ are linear functions of their argument.
Indeed, writing once again
\begin{equation}\label{eqn:main_multi1}
\bX = \bM + \sum_{i=1}^\ell f_i(\bX) \bN_{i}, \quad {\bM={\cal L}^{-1}(\bD),}  \,\,\bN_{i} = -{\cal L}^{-1}(\bC_i), 
\end{equation}
we can compute
$$
f_j(\bX) = f_j(\bM) + \sum_{i=1}^\ell f_i(\bX) f_j(\bN_{i}), \qquad j=1, \ldots, \ell .
$$
Let $\sigma_j = f_j(\bX)$.
Collecting all quantities, we obtain the $\ell\times \ell$ linear system
{\footnotesize
\begin{equation}\label{eq:sysf}
\begin{bmatrix}
1-f_1(\bN_{1}) & -f_1(\bN_{2}) & \cdots & -f_1(\bN_{\ell}) \\
-f_2(\bN_{1})  & 1-f_2(\bN_{2}) & \cdots & -f_2(\bN_{\ell}) \\
\vdots & \vdots & \ddots & \vdots \\
-f_\ell(\bN_{1}) &  \cdots & \cdots & 1-f_\ell(\bN_{\ell}) 
\end{bmatrix}
\begin{bmatrix}
\sigma_1 \\ \vdots \\ \sigma_\ell
\end{bmatrix}
= 
\begin{bmatrix}
f_1(\bM) \\ \vdots \\ f_\ell(\bM) 
\end{bmatrix} \,\, \Leftrightarrow \,\, (\bI-\bF) {\bm \sigma} = \bbf,
\end{equation}
}
where $\bI$ is the identity matrix of matching dimensions.
Solving this small linear system yields the coefficients in
$$
\bX = \bM + \sum_{i=1}^\ell \sigma_i \bN_{i},
$$
which generalizes the formula in Proposition~\ref{prop:f(X)}.
In general, the cost of solving this system remains moderate compared
with all other computational costs as long as $\ell$ is significantly lower than $n$.
Clearly, the solution uniqueness is related to the nonsingularity of $\bI-\bF$.
A well known sufficient condition for the nonsingularity is that $\|\bF\| < 1$ {where
$\|\cdot\|$ is any induced matrix norm}.

\subsection{An application to solid mechanics}\label{sec:appl}
{The modelling of masonry-like materials
calls for the computation of the projection of a symmetric matrix onto the cone of negative semidefinite 
symmetric matrices with respect to the inner product defined by an assigned positive
definite symmetric linear map ${\cal C}$, associating $n\times n$ symmetric matrices with $n\times n$ symmetric matrices.
%
The map $\cal C$ contains the mechanical properties of the masonry material and can take different forms depending on the anisotropy of the material. 
When $\cal C$ models the elasticity tensor of an isotropic elastic material,
it takes the form
\begin{equation}\label{CI}
{\cal C} (\bX) = \frac{E}{1+\nu} \left (\bX +\frac{\nu}{1-2\nu} \trace(\bX) \bI\right ),
\end{equation}
where
$E$ is Young's modulus, $E>0$, and $\nu$ is the Poisson ratio, satisfying $\nu \in (-1,1/2)$. 
When, on the other hand, $\cal C$ represents a transversely isotropic elasticity tensor with respect to the direction 
$\mathbf{e}_{3}$, then it can be written as
%
${\cal C} (\bX) =\sum_{i=1}^{\ell}  {\trace(\bH_i \bX) \bK_i}$,
for $\ell = n(n+1)/2$ and suitable symmetric matrices $\bH_i, \bK_i \in \RR^{n\times n}$
for $i = 1, \dots, \ell$ which depend on the scalars $E$ and $\nu$ and on the spectral representation of $\cal C$
\cite{Padovani.02}.

For a given symmetric matrix $\bar \bY$, in \cite{Padovani.Porcelli.22} the projection problem was reformulated as the following  quadratic
semidefinite programming problem 
\begin{equation}\label{sdp_primal}
  \begin{array}{ll}
     \min_{\bY} & \trace(\bY {\cal C}(\bY + \bar \bY)) \\
     \mbox{s.t. } &  \bY \succeq 0, \\
    \end{array}
\end{equation}
and a primal-dual path-following interior point method was proposed. At each iteration of the interior-point method, one Newton step is computed for the following perturbed first-order optimality conditions for problem (\ref{sdp_primal}) 
\begin{equation}\label{F_mu}
F_{\mu}(\bY,\bS)=  \left (\begin{array}{c}
 \bS -{\cal C}(\bY +\bar \bY)\\
 \bY \bS- \mu \bI
\end{array}
\right )=\mathbf{0}, \qquad \bY \succ 0,\  \bS \succ 0,
\end{equation}
where the positive scalar $\mu$ is driven to zero as the method progresses. 
To ensure that the Newton steps produce symmetric matrices, 
different symmetrization schemes can be applied to the nonlinear equation $\bY \bS- \mu \bI = 0$ in (\ref{F_mu}):
the popular Alizadeh-Haeberly-Overton (AHO) and Nesterov-Todd (NT) schemes 
have been explored in \cite{Padovani.Porcelli.22}. Fixed $\mu>0$ and given the current approximation
$(\bY, \bS)$ of the solution of (\ref{F_mu}), let $\bX$ denote the Newton step for the variable $\bY$. 
Consider first the AHO scheme: $\bX$ solves the equation
\begin{equation}\label{eqn:eqnC}
\bS \bX + \bX \bS + {\cal C} (\bX) \bY + \bY {\cal C} (\bX) = \bD,
\end{equation}
where the right-hand side $\bD = 2\mu \bI-(\bY\bS+\bS\bY)-(\bY({\cal C}(\bY+\bar \bY) -\bS) + ({\cal C}(\bY+\bar \bY) -\bS)\bY )$ takes into account the
value of the current $F_{\mu}(\bY,\bS)$ and the AHO symmetization.
When $\cal C$ is isotropic, inserting the form (\ref{CI}) into (\ref{eqn:eqnC}) yields
$$
\left (\bS + \frac{E}{1+\nu} \bY \right) \bX + \bX \left(\bS + \frac{E}{1+\nu} \bY\right) + 
 \trace(\bX) \frac{\nu E}{(1+\nu)(1-2\nu)}\bY  = \bD,
$$
that corresponds to (\ref{eqn:main}) with $\bA = \bB = \left (\bS + \frac{E}{1+\nu} \bY \right)$ 
and $\bC = \frac{\nu E}{(1+\nu)(1-2\nu)}\bY$.
%
%
If $\cal C$ is transversely isotropic, the terms involving ${\cal C}$ in (\ref{eqn:eqnC}) are given by
$$
{\cal C} (\bX) \bY + \bY {\cal C} (\bX) = \sum_{i=1}^{ \ell}  \trace(\bH_i \bX) (\bK_i \bY + \bY \bK_i) \equiv 
\sum_{i=1}^{ \ell} f_i(\bX) \bC_i,
$$
yielding 
$$
\bA \bX + \bX \bA + \sum_{i=1}^{ \ell} f_i(\bX) \bC_i = \bD,
$$
with $\bA = \bS $, which thus corresponds to (\ref{eqn:main_multi}).

In the case of the NT scheme, the Newton step solves the general equation
$\bW \bX \bW + {\cal C} (\bX)  = \bD$ with $\bW \succ 0$ being the geometric mean of $\bY^{-1}$ and $\bS$,
 and $\bD$ is suitably defined taking into account the residual $F_{\mu}(\bY,\bS)$ and the NT scheme, see e.g. \cite{Todd.Tutuncu.Toh.07}.
If $\cal C$ is isotropic, the equation above reads
$$
\bW \bX \bW + \frac{E}{1+\nu} \bX + \frac{\nu E}{(1+\nu)(1-2\nu)}\trace(\bX) \bI = \bD.
$$
Dividing by $\bW$,
$$
\bX \bW + \frac{E}{1+\nu}\bW^{-1} \bX + \trace(\bX) \frac{\nu E}{(1+\nu)(1-2\nu)} \bW^{-1} = \bW^{-1}\bD,
$$
that is in the form (\ref{eqn:main}) with $\bA = \frac{E}{1+\nu}\bW^{-1}$,  $\bB = \bW$ 
and $\bC = \frac{\nu E}{(1+\nu)(1-2\nu)} \bW^{-1}$. 
Finally, for the transversely isotropic case one obtains the equation
$$
\bX = \bM + \sum_{i=1}^{ \ell} f_i(\bX) \bN_i, \quad 
\bM = \bW^{-1} \bD \bW^{-1}, \,\, \bN_i =- \bW^{-1} \bK_i \bW^{-1},
$$
which has the form (\ref{eqn:main_multi1}), with $f_i(\bX)=  \trace(\bH_i \bX), i=1, \dots, \ell$.

We remark that the explicit form of Newton step above within the NT scheme is a generalization
of the formula given in  \cite[Lemma 5.1]{Todd.Tutuncu.Toh.07} 
for the case ${\cal C}(\bX) = \bK \bX \bK $ and $\bK \succeq 0$.


\section{The trace of a matrix power}\label{sec:trace_power}
A first generalization to the nonlinear setting is given by
the family of functions $f(\bX)=  \trace(\bX^p)$, with $p\in\NN$, $p>1$.
For moderate $p$ such as $p=2$, it is possible to give explicit solutions to the
problem.
%
We focus on the effect of
$f$ on the matrix equation, where we work with the form in (\ref{eqn:main1}). 

Let $p=2$. We have
\begin{eqnarray*}
f(\bX)&=&  \trace(\bX^2)= \trace( (\bM +  f(\bX)\bN) (\bM +  f(\bX)\bN) ) \\
&=&
\trace(\bM^2)+  2\,\trace(\bM \bN) f(\bX) + f(\bX)^2 \trace(\bN^2) \\
&=&
f(\bM)+  2\,\trace(\bM \bN) f(\bX) + f(\bX)^2 f(\bN).
\end{eqnarray*}
Let $\beta = 2\,\trace(\bM \bN) -1$.
The equation above corresponds to the following (scalar) quadratic algebraic equation in the variable 
 $r=f(\bX)$,
$$
r^2 f(\bN)  +  \beta r + f(\bM) =0 .
$$
If $f(\bN)=0$ and $\beta\ne 0$ then the solution is $r=-f(\bM)/\beta$,
giving $\bX =  \bM + r \bN$.
If $f(\bN)\ne 0$ then the following two solutions are derived,
$$
r_{1,2} = \frac {1}{2 f(\bN)} \left( -\beta \pm \sqrt{\beta^2 - 4 f(\bN) f(\bM)}\right).
$$
The two final solution matrices $\bX_{(1)}, \bX_{(2)}$ are obtained as 
$$
 \bX_{(1)} =  \bM + r_{1} \bN ,
\qquad
\bX_{(2)} =  \bM + r_{2} \bN .
$$

For higher powers of $\bX$, correspondingly larger degree scalar polynomial equations are obtained,
from which the corresponding {\it numerical} solution matrices can be derived, in case the
roots can only be computed numerically. 
The procedure may also yield complex (conjugate) values for $r$ even for real data, 
from which complex (conjugate) solutions will follow.

Powers of affine functions can also be
considered, such as $f(\bX)=\trace( (\bX+\bH)^p)$, for a fixed matrix $\bH$.
A similar solution procedure can be devised for other, related functions such as the Frobenius norm, that is
$$
f(\bX) = \|\bX\|_F^2 = \trace(\bX^T \bX).
$$

A second generalization for which explicit solutions can be obtained under
certain hypotheses is the function $f(\bX)= \trace(\bX^{-1})$. 

\begin{proposition}\label{prop:inv}
Let $\bM=\bmm_1 \bmm_2^T$ be a rank-one matrix and $\bN$ be invertible.
Let the nonlinear function be $f(\bX)= \trace(\bX^{-1})$. 
If the matrix equation $\bX=\bM + f(\bX)\bN$ admits nonsingular solutions,
then these solutions  are given
as $X_{(i)} = \bM + r_i \bN$, $i=1,\ldots, 3$ where $r_i$ are the roots of the
polynomial equation
$$
r^3 + \eta_2 r^2 + \eta_1 r + \eta_0 = 0, 
$$
with
$\eta_2=\bmm_2^T \bN^{-1}\bmm_1$, $\eta_1=-f(\bN)$ and $\eta_0 = \eta_1 \eta_2 + \bmm_2^T \bN^{-2}\bmm_1$. 
\end{proposition}

\begin{proof}
We first note that
if $\bX$ is a nonsingular solution to the given equation, then $f(\bX)\ne 0$ must hold,
otherwise $\bX = \bM$ would not be invertible. Using the
Sherman-Morrison formula we obtain
\begin{eqnarray*}
\bX^{-1} &=& (\bmm_1 \bmm_2^T + f(\bX) \bN)^{-1} \\
&=&
\frac 1 {f(\bX)} 
\left ( \bN^{-1} - \bN^{-1} \bmm_1 ( f(\bX)+ \bmm_2^T \bN^{-1} \bmm_1)^{-1} \bmm_2^T \bN^{-1}\right ).
\end{eqnarray*}
Using $f(\bX) = \trace( (\bmm_1 \bmm_2^T + f(\bX) \bN)^{-1})$, we obtain
$$
f(\bX) = \frac 1 {f(\bX)}
\left ( f(\bN) - ( f(\bX) + \bmm_2^T \bN^{-1}\bmm_1)^{-1} \bmm_2^T \bN^{-1}  \bN^{-1} \bmm_1 \right ).
$$
Reordering terms, the third degree polynomial in $r=f(\bX)$ is obtained.
\end{proof}

A similar result can be obtained by exchanging the role of $\bM$ and $\bN$, that is,
requiring that $\bN$ is rank-one and $\bM$ nonsingular, giving rise to at most
two distinct solutions.
More precisely, given the problem $\bX = \bM+f(\bX) \bn_1 \bn_2^T$, similar algebraic
steps show that the solutions
are given as $X_{(i)} = \bM+ r_i \bn_1 \bn_2^T$, where $r_i$ are the roots of
the polynomial 
$$
\eta_2 r^2 + \eta_1 r + \eta_0 = 0, 
$$
with $\eta_0 = -f(\bM)$,  $\eta_2=\bn_2^T \bM^{-1}\bn_1$ and $\eta_1=1+\eta_0 \eta_2 + \bn_2^T\bM^{-2}\bn_1$.

If $\bM$ or $\bN$ have larger rank, then the procedure described above cannot
be directly generalized.

\begin{example}\label{ex:rankone}
In Figure~\ref{fig:rankone}
we give a general test code in Matlab \cite{matlab7} for the case of $f(\bX)={\rm trace}(\bX^{-1})$
and the solution formulas in Proposition~\ref{prop:inv} and the subsequent discussion.
The left hand side implements the case of $\bM$ rank-one, while the right-hand side
refers to the case of $\bN$ rank-one. The obtained computational results are
\vskip 0.1in

{\tt [1.5543e-15   5.0626e-14   2.4425e-15]  \hskip 0.2in [8.3313e-16   8.3313e-16]}

\begin{figure}
{\scriptsize
\begin{verbatim}
f=@(X)(trace(inv(X)));
n=10; rng(2)                                           rng(1)
%X=m1*m2'+f(X) N;                                      %X=M+f(X) n1*n2';
m1=randn(n,1); m2=randn(n,1);                          n1=randn(n,1); n2=randn(n,1);
N=randn(n,n);                                          M=randn(n,n);
 
t2=m2'/N*m1;                                           f2=n2'/M*n1;
t1=-trace(inv(N));                                     f1=1-f(M)*f2+ n2'/M^2*n1;
t0=t1*t2+m2'/N^2*m1;                                   f0=-f(M);

r=roots([1 t2 t1 t0]);                                 r=roots([f2, f1, f0]);

X1=m1*m2'+r(1)*N;                                      X1=M+r(1)*n1*n2';
X2=m1*m2'+r(2)*N;                                      X2=M+r(2)*n1*n2';
X3=m1*m2'+r(3)*N;

[norm(f(X1)-r(1)) norm(f(X2)-r(2)) norm(f(X3)-r(3))]   [norm(f(X1)-r(1)), norm(f(X2)-r(2))]
\end{verbatim}
}
\caption{Matlab code for Example \ref{ex:rankone}.\label{fig:rankone}}
\end{figure}

\end{example}

\section{The nonlinear case. Linear-nonlinear composition}\label{sec:lin_nonlin}
The problem changes significantly in case the function $f$ has the general form
$$
f(\bX) = \phi(\psi(\bX)), \quad 
\phi: \RR^{n\times n} \to \RR, \quad \psi: \RR^{n\times n} \to \RR^{n\times n} ,
$$
where $\phi$ is linear,  and $\psi$ is a {(nonlinear)} matrix function \cite{Higham2008}.
This is the case for instance for $f(\bX)={\rm trace}(\exp(-\bX))$.
We focus on the small size  case, and use the form in (\ref{eqn:main1}) derived
from (\ref{eqn:main}) after the application of the inverse Sylvester operator.

Let us consider the case when $\phi(\bY) = {\rm trace}(\bY)$, and assume that $\bN$
is diagonalizable, so that  $\bN=\bQ\bLambda \bQ^{-1}$.  Then
(\ref{eqn:main1}) is equivalent to
$$
 \bQ^{-1}\XX \bQ =   \bQ^{-1}\bM \bQ + f(\XX) \bLambda .
$$
We then note that
$$
f(\XX) = {\rm trace}(\psi(\XX))= {\rm trace}(\psi(\bQ^{-1}\XX \bQ))=f(\bQ^{-1}\XX \bQ),
$$
as the trace is invariant under similarity transformations.
Let $\bX_1 \equiv \bQ^{-1}\XX \bQ$ and $\bM_1 = \bQ^{-1}\bM \bQ$, so that
\begin{eqnarray}\label{eqn:linmat}
 \bX_1  =\bM_1 +   f(\bX_1) \bLambda.
\end{eqnarray}
This form shows that the scalar value $f(\bX_1)$ only appears in the diagonal elements
of $\bX_1$, while the off-diagonal part of $\bX_1$ coincides with $\bM_1$. 
In spite of this simple relation, it is hard to determine expressions for the
solution in closed form for general matrix functions $\psi$, since 
computing ${\rm trace}(\psi(\bX_1))$
still involves the whole matrix $\bX_1$, and the nonlinearity of $\psi$ does not
allow for algebraic simplifications. We are thus led to consider classical
iterative schemes for solving (\ref{eqn:linmat}). 

Using the formulation in (\ref{eqn:linmat}),  starting with some $\bX_1^{(0)}$, a fixed point iteration can be written as
\begin{equation}\label{eqn:sequence}
 \bX_1^{(k+1)}  =\bM_1 +   f(\bX_1^{(k)}) \bLambda,
\end{equation}
for $k\ge0$, where, it is apparent that only the diagonal elements of $\bX_1^{(k+1)}$
are updated at each iteration $k$, while
the off-diagonal elements of $\bX_1$ still coincide with those of $\bM_1$ and
they never change through the iteration. Combining two consecutive iterations, we obtain
for $k\ge1$,
$$
 \bX_1^{(k+1)}  =\bX_1^{(k)} +   (f(\bX_1^{(k)})-f(\bX_1^{(k-1)})) \bLambda,
$$
which only updates the diagonal elements of the matrices. Hence, setting $\blambda={\rm diag}(\bLambda)$
where the function diag extracts the diagonal elements of a matrix, we can write
\begin{equation}\label{eqn:diag}
 {\rm diag}(\bX_1^{(k+1)})  ={\rm diag}(\bX_1^{(k)}) +   (f(\bX_1^{(k)})-f(\bX_1^{(k-1)})) \blambda .
\end{equation}

%
%

The final solution is  obtained as $\bX^{(k)} = \bQ \bX_1^{(k)} \bQ^{-1}$.
We observe that for $\bN$ non-symmetric, the conditioning of
$\bQ$ influences the error norm in the final 
approximate solution.  More precisely,
let $X_1^{\star}$ be the exact solution to (\ref{eqn:linmat}). Then
 $\|\bX^{(k)} -\bX^\star\|\le \|\bQ\|\,\|\bQ^{-1}\| 
\|\bX_1^{(k)} -\bX_1^\star\|$, so that
the final $\bX^{(k)}$ may be less accurate than the iteration in $\bX_1^{(k)}$ would grant.


We first report on an algebraic characterization of the iteration, and then focus on
error norm bounds
for a selection of well known matrix functions. To this end, we will focus on the form (\ref{eqn:linmat}), 
in which only the diagonal elements are modified by
the iteration, when taking $\bX_1^{(0)}=\bM_1$. The same occurs for the error matrix.

\begin{proposition}\label{prop:sign}
Let $\bX_1^{(0)}=\bM_1$ and $\{\bX_1^{(k)}\}_{k\ge 1}$ be the sequence of iterates from 
(\ref{eqn:linmat}), with $\bM_1\succ 0$ and $\bLambda\succ 0$.

i) If $f$ is a nonnegative function satisfying $f(\bX)-f(\bY)\le 0$ for $\bY-\bX\succeq  0$,
then $\bX_1^{(k+1)}\succeq \bX_1^{(k)}$ for all $k$s.

ii) If $f$ is a nonnegative function satisfying $f(\bX)-f(\bY)\ge 0$ for $\bY-\bX\succeq  0$,
then the  iterates $\bX_1^{(k+1)}-\bX_1^{(k)}$ alternate definiteness at each $k$.
\end{proposition}

\begin{proof}
For $k=0$, $\bX_1^{(1)} = \bM_1 + f(\bM_1)\bLambda \succeq \bM_1=\bX_1^{(0)}$.
For the subsequent iterates we have
$\bX_1^{(k+1)} = \bM_1 + f(\bX_1^{(k)})\bLambda$ and
$\bX_1^{(k)} = \bM_1 + f(\bX_1^{(k-1)})\bLambda$. 
Subtracting we obtain
$$
\bX_1^{(k+1)}-\bX_1^{(k)} = (f(\bX_1^{(k)}) - f(\bX_1^{(k-1)}) ) \bLambda.
$$
Hence, for any $k>1$, 
if $f(\bX_1^{(k)}) - f(\bX_1^{(k-1)})>0$ then $\bX_1^{(k+1)}-\bX_1^{(k)}\succeq 0$ 
because $\bLambda \succeq 0$; this shows (i).
For (ii),
if $\bX_1^{(k)} - \bX_1^{(k-1)} \succeq 0$ then $(f(\bX_1^{(k)}) - f(\bX_1^{(k-1)}) ) <0$ so that
$\bX_1^{(k+1)}-\bX_1^{(k)} \preceq 0$, and viceversa.
\end{proof}

Recalling relation (\ref{eqn:diag}), the definiteness explored in Proposition~\ref{prop:sign}
refers to the way the diagonal elements of the iteration matrix change.
In the first case, these elements grow monotonically as the iterations proceed; in case of
convergence, the diagonal elements reach the final value from below.
On the other hand, if $f$ grows monotonically, the diagonal entries may
showcase an alternating leapfrog behavior, which in case of convergence will
terminate with the exact solution.

For instance, the function $\trace(\bX^{1/2})$ satisfies (i) while
the function $\trace(\exp(-\bX))$ satisfies (ii). The different behavior is reported in
Figure~\ref{fig:signs} for the iteration in (\ref{eqn:diag}), with the data created below in 
Matlab \cite{matlab7}.
The values correspond to the $(n/2,n/2)$ diagonal element, however all diagonal elements
behave similarly, as they change by the same factor.

\vskip 0.1in
{\footnotesize
\begin{verbatim}
n=10; rng(1)
N=randn(n,n); N=sqrtm(N'*N);                             N=rand(n,n); N=0.2*sqrtm(N'*N);
f=@(X)(trace(expm(-X)));                                 f=@(X)(trace(sqrtm(X)))
Xstar=2*n*randn(n,n); Xstar=sqrtm(Xstar'*Xstar);
M=Xstar-f(Xstar)*N;
\end{verbatim}
}
\vskip 0.1in

\begin{figure}[tbhp]
 \centering
\includegraphics[height=.35\textwidth,width=0.48\textwidth]{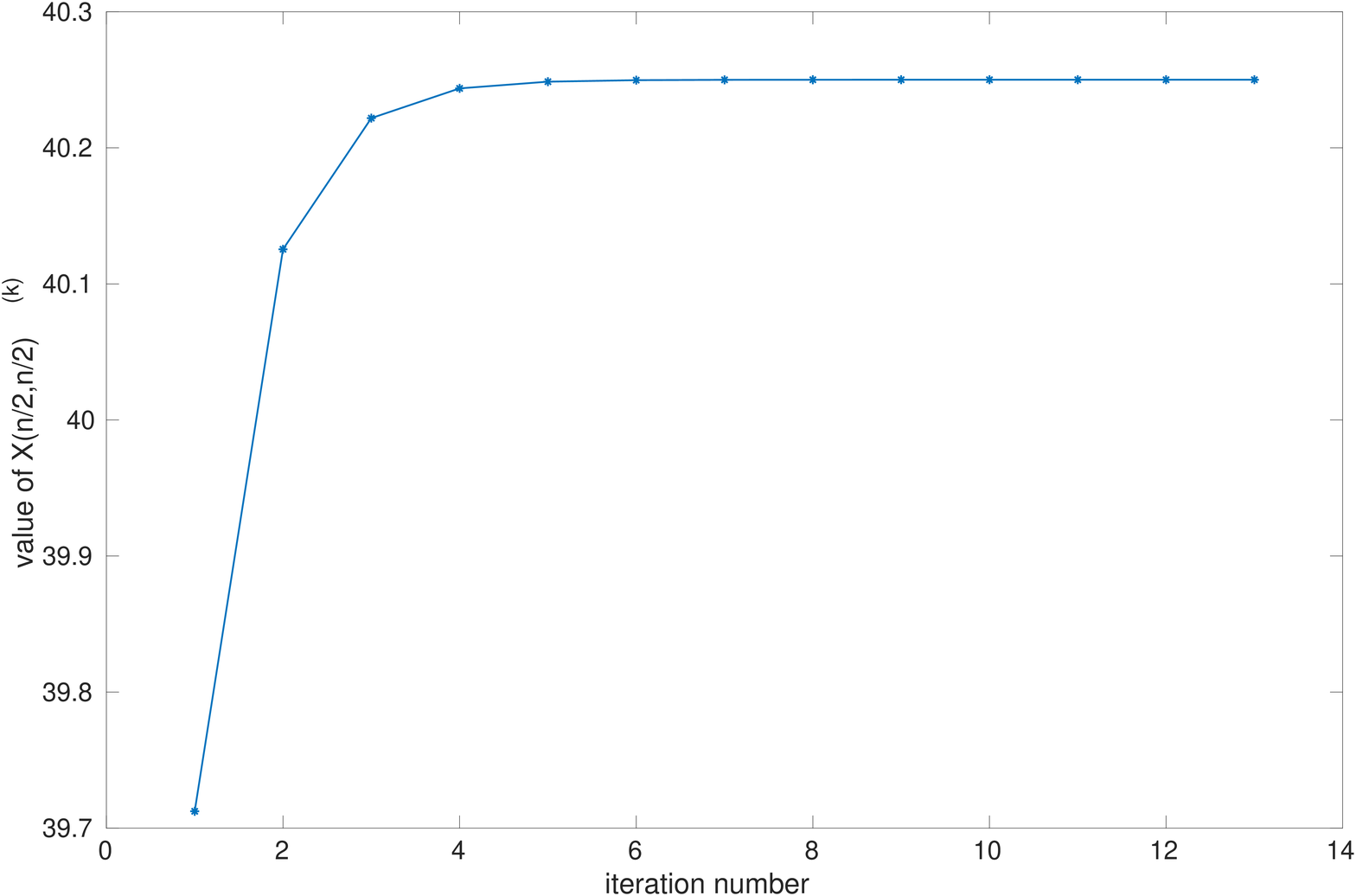}
\includegraphics[height=.35\textwidth,width=0.48\textwidth]{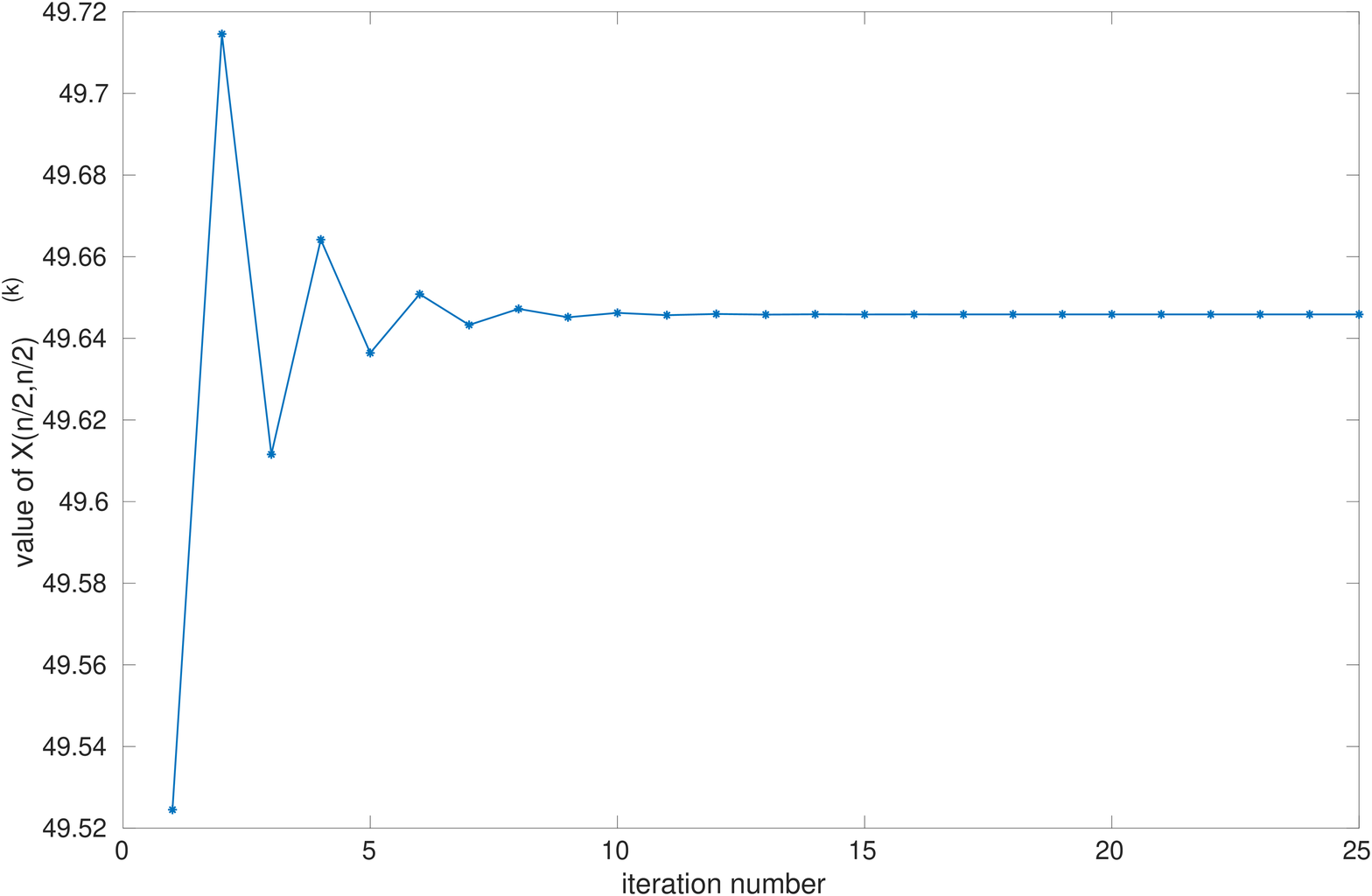}
  \caption{Convergence of the $(n/2, n/2)$ diagonal element of $\bX_1^{(k)}$.
Left: $f(\bX)=\trace(\bX^{1/2})$.  Right: $f(\bX)=\trace(\exp(-\bX))$.  \label{fig:signs}}
\end{figure}

As a first consideration for the convergence analysis of the iteration $\{\bX^{(k)}\}_{k\ge 0}$, 
we first notice that
\begin{equation}\label{eqn:error_eta}
 \bE^{(k+1)}\equiv \bX^{(k+1)}_{1}-\bX_1^\star = (f(\bX^{(k)}_{1})-f(\bX_1^\star) ) \bLambda\equiv \eta_k \bLambda,
\end{equation}
showing that the error is a scalar multiple of a constant, diagonal matrix, and
only the scalar $\eta_k$ changes with $k$. Moreover, for any matrix norm $\|\cdot \|$,
\begin{eqnarray*}
 \|\bX_1^{(k+1)} -\bX_1^{\star}\|  &= &
| (f(\bX_1^{\star})-f(\bX_1^{(k)}))|\,  \|\bLambda\|  \\
&=&
\left (\frac{| (f(\bX_1^\star)-f(\bX_1^{(k)}))|}{\|\bX_1^{(k)} -\bX_1^\star\|}
\, \|\bLambda\| \right )
 \|\bX_1^{(k)} -\bX_1^\star\|. 
\end{eqnarray*}
The quantity in parenthesis is what we expect from the scalar case.
Due to the linearity of the trace, the quotient in parenthesis is closely related
to the differential of the considered matrix function $\psi$.
To make more precise statements about convergence we thus need to focus on specific examples of $f$.
We will make use of the
Frechet derivative of a matrix function $\psi$, defined as the linear function $L(\bX, \bE)$ such that
we have $\psi(\bX+\bE)-\psi(\bX)=L(\bX,\bE)+o(\|\bE\|)$ for $\|\bE\|$ sufficiently
small \cite[section 3.1]{Higham2008}.
If the Frechet derivative of the given function $\psi$ in $\bX$ exists, then $\psi$ is said to
be Frechet differentiable at $\bX$.

\begin{theorem}\label{th:errorbound}
Let $f(\bX)={\rm trace}(\exp(-\bX))$, and let 
$\{\bX_1^{(k)}\}_{k\ge 0}$ be the sequence of 
iterates from (\ref{eqn:sequence}), with $\bM_1 \succeq 0$ and $\bLambda\succeq 0$.
If $\bX_1^\star$ is a solution to (\ref{eqn:linmat}) and
we let $\bE^{(k)}=\bX_1^{(k)}-\bX_1^\star$, then 
for $\bX_1^{(k)}$ sufficiently close to $\bX_1^\star$ we have
$$
\bE^{(k+1)}  = {\rm trace}(\bLambda \exp(-\bX_1^\star)) \bE^{(k)} + o(\|\bE^{(k)}\|)\bLambda .
$$
\end{theorem}

\begin{proof}
The matrix exponential is Frechet differentiable at any $\bX$, and its Frechet derivative
is given by 
$L(\bX,\bE) = \int_0^1 \exp(-\bX(1-s))\bE \exp(-\bX s) ds$ \cite[formula (10.15)]{Higham2008}.

The differential of $f$ thus corresponds to ${\rm trace}(L(\bX,\bE))$, which can be written as
\begin{eqnarray*}
{\rm trace}(L(\bX,\bE)) &=& \int_0^1 {\rm trace}( \exp(-\bX(1-s))\bE \exp(-\bX s)) ds \\
&=&
\int_0^1 {\rm trace}(\bE \exp(-\bX s)\exp(-\bX(1-s)) ) ds \\
&=&
\int_0^1 {\rm trace}(\bE \exp(-\bX)) ds = {\rm trace}(\bE \exp(-\bX)).
\end{eqnarray*}

Hence, for  $\bX_1^{(k)}$ 
sufficiently close to $\bX_1^\star$, using (\ref{eqn:error_eta}) we have
\begin{eqnarray*} \label{eqn:errorsequence}
\bX_1^{(k+1)}-\bX_1^\star &=& (f(\bX_1^{(k)})-f(\bX_1^\star) ) \bLambda \\
&=& {\rm trace}(L(\bX_1^\star,\bE^{(k)})) \bLambda + o(\|\bE^{(k)}\|) \bLambda \\
&=&  {\rm trace}(\bE^{(k)} \exp(-\bX_1^\star)) \bLambda + o(\|\bE^{(k)}\|) \bLambda \\
&=&  {\rm trace}(\bLambda \exp(-\bX_1^\star)) \eta_{k-1}\bLambda + o(\|\bE^{(k)}\|) \bLambda \\
&=&  {\rm trace}(\bLambda \exp(-\bX_1^\star)) \bE^{(k)} + o(\|\bE^{(k)}\|) \bLambda,
\end{eqnarray*}
and the proof is complete.
\end{proof}

The above expression for the error allows us to give a sufficient condition for convergence.
The proof follows the usual steps of the Ostrowski Theorem; see, e.g., \cite[10.1.3]{Ortega.Rheinboldt.00}.

\begin{theorem}\label{th:errorbound2}
Assume that the notation and hypotheses of Theorem \ref{th:errorbound} hold.
Suppose that $\psi$ is Frechet differentiable at $\bX_1^\star$. If
${\rm trace}(\bLambda \exp(-\bX_1^\star))=\sigma <1$ then there exist 
an $\bX_1^{(0)}$ and a  $\sigma_1 \in [0,1)$ such that
$$
\|\bE^{(k+1)}\| \le \sigma_1 \|\bE^{(k)}\| ,
$$
for $k\ge0$, for any matrix norm $\|\cdot \|$.
\end{theorem}

\begin{proof}
We proceed by induction.
The differentiability of $\psi$ and Theorem~\ref{th:errorbound}
ensure that for an arbitrary  $\epsilon >0$ there exists a $\bX_1^{(0)}$ sufficiently close to $\bX^\star$
such that
$$
||\bE^{(1)}  - {\rm trace}(\bLambda \exp(-\bX_1^\star)) \bE^{(0)}\| \le \epsilon \|\bE^{(0)}\|,
$$
so that
\begin{eqnarray*}
||\bE^{(1)}\| &\le & ||\bE^{(1)}  - {\rm trace}(\bLambda \exp(-\bX_1^\star)) \bE^{(0)}\| +
  |{\rm trace}(\bLambda \exp(-\bX_1^\star))|\, \|\bE^{(0)}\| \\
& \le & (\epsilon + \sigma) \|\bE^{(0)}\|.
\end{eqnarray*}
By taking $\epsilon$ so that $\sigma_1 = \epsilon + \sigma <1$ the result follows for $k+1=1$.
Assuming now that the result holds for $\|\bE^{(k)}\|$, we can write again
$$
||\bE^{(k+1)}  - {\rm trace}(\bLambda \exp(-\bX_1^\star)) \bE^{(k)}\| \le \epsilon \|\bE^{(k)}\|,
$$
and proceed as for $k=0$, to obtain the final bound.
\end{proof}


\begin{example}\label{ex:fixed_ex1}
We analyze the convergence of the fixed point iteration with respect to the 
condition of Theorem~\ref{th:errorbound2} on the derivative
${\rm trace}(\bLambda \exp(-\bX_1^\star))$. 
To this end, we consider $f(\bX)={\rm trace}(\exp(-\bX))$ and the matrix $\bX^\star = \sqrt{\alpha} \bG$ with 
$\bG = (\bG_0^T \bG_0)^{\frac 1 2}$, where
$\bG_0$ ={\tt randn(n,n)} (Matlab seed {\tt rng(1)}).
By varying  $\alpha$ a different magnitude of the Frechet derivative can be obtained.
The matrix $\bN$ is defined in the same way as $\bG$, and $\bM=\bX^\star - f(\bX^\star)\bN$,
In Table~\ref{tab:ex1} we report the
results of the fixed point iteration $\bX_1^{(0)}=\bM_1$, $\bX_1^{(k+1)} = \bM_1 + f(\bX_1^{(k)}) \bLambda$, $k=0, 1, \ldots$.
The iteration stops either for $\|\bX^{(k+1)}- (\bM + f(\bX^{(k+1)}) \bN)\|/\|\bM\| < 10^{-7}$ or
for $k=500$. The numbers in the table show lack of convergence as soon as
the condition on the derivative fails, as is typical of Ostrowski type theorems.
\end{example}

\begin{table}
\centering
\begin{tabular}{ccrc}
  ${\rm trace}(\bLambda \exp(-\bX_1^\star))$  & $\alpha$ & $k$ &  $\frac{\|\bX^{(k+1)}- (\bM + f(\bX^{(k+1)}) \bN)\|}{\|\bM\|}$ \\
\hline
   0.079 & 12.589 & 3  &8.3190e-08   \\
   0.176 & 10.000 & 6  &3.4123e-08   \\
   0.335 & 7.9433 & 11 & 3.7944e-08   \\
   0.570 & 6.3096 & 23 & 6.9902e-08   \\
   0.889 & 5.0119 & 117&  9.6324e-08   \\
   1.296 & 3.9811 & 500&  3.5943e-01   \\
   1.789 & 3.1623 & 500&  1.2832e+00   \\
\hline
\end{tabular}
\caption{Example \ref{ex:fixed_ex1}.
$f(\bX)={\rm trace}(\exp(-\bX))$. Behavior of the iteration (\ref{eqn:diag}) as $\alpha$ varies. \label{tab:ex1}}
\end{table}

We next derive similar results for a nonlinear function involving the matrix square root.

\begin{theorem}
Let $f(\bX)={\rm trace}(\bX^{\frac 1 2})$, and let 
$\{\bX_1^{(k)}\}_{k\ge 0}$ be the sequence of iterates 
from (\ref{eqn:linmat}), with $\bM_1\succeq 0$ and $\bLambda\succeq 0$,
so that the exact solution $\bX_1^\star$ to (\ref{eqn:linmat})
is symmetric and positive definite.
Let $\bE^{(k)}=\bX_1^{(k)}-\bX_1^\star$.
Then
$$
\bE^{(k+1)}  = \frac 1 2 {\rm trace}(\bLambda (\bX_1^\star)^{-1}) \bE^{(k)} + o(\|\bE^{(k)}\|)\bLambda.
$$
If $\frac 1 2 {\rm trace}(\bLambda (\bX_1^\star)^{-1})=\sigma<1$ then 
there exist an $\bX_1^{(0)}$ and a
$\sigma_1 \in [0,1)$ such that
$$
\|\bE^{(k+1)}\| \le \sigma_1 \|\bE^{(k)}\|  ,
$$
for any matrix norm $\|\cdot \|$.
\end{theorem}

\begin{proof}
For the matrix square root we have
$L(\bX,\bE) = {\cal L}_{\bX}^{-1}(\bE)$, where ${\cal L}_{\bX}$ is the linear operator
${\cal L}_{\bX}  : \bZ \mapsto \bX \bZ + \bZ \bX$ \cite[p.134]{Higham2008}. 
Using Remark~\ref{rem:solve} and (\ref{eqn:error_eta}) we can write
\begin{eqnarray*}
\bX_1^{(k+1)}-\bX_1^\star  &=& 
{\rm trace}({\cal L}_{\bX_1^\star}^{-1}(\bE^{(k)}))\bLambda + o (\|\bE^{(k)}\|) \bLambda \\
&=& \frac 1 2 {\rm trace}((\bX_1^\star)^{-1}\bE^{(k)})\bLambda + o (\|\bE^{(k)}\|) \bLambda \\
&=& \frac 1 2 {\rm trace}((\bX_1^\star)^{-1}\bLambda) \bE^{(k)} + o (\|\bE^{(k)}\|) \bLambda,
\end{eqnarray*}
and the first result follows. 
The proof of the final bound follows the same lines as the corresponding
bound in Theorem~\ref{th:errorbound2}.
\end{proof}


\section{The nonlinear case. Nonlinear-linear composition}\label{sec:nonlin_lin}
The procedure described in Proposition~\ref{prop:f(X)} can be employed
within a procedure for solving (\ref{eqn:main}) when the nonlinear function has the form
$f(\bX) = g(h(\bX))$ where $g : [\alpha,\beta] \to \RR$ and
$h$ is a real valued {\it linear} function with image in $[\alpha,\beta]$.
The problem becomes nonlinear in $\bX$, hence
uniqueness of the solution is in general not guaranteed.

To analyze the new setting, consider again equation (\ref{eqn:main1}), that is
$\bX = \bM + f(\bX) \bN$,
and apply the linear function $h$ to both sides,
\begin{equation}\label{eqn:phi}
h(\bX) = h(\bM) + f(\bX)h(\bN).
\end{equation}
For $\gamma_1\equiv h(\bM)$, $\gamma_2\equiv h(\bN)$ and
setting $y \equiv h(\bX)$, the equation above corresponds to the nonlinear
scalar equation
\begin{equation}\label{eqn:f_nonlin}
 \gamma_1 + g(y)\gamma_2 - y = 0, \qquad y \in [\alpha,\beta] .
\end{equation}
We next formalize the fact that
if this equation has a solution $y^*$ in the considered interval, then (\ref{eqn:f_nonlin}) yields 
a solution to (\ref{eqn:main1}). To make the treatment simpler, we assume that
$h(\bX)={\rm trace}(\bX)$.  The general case
$h(\bX)={\rm trace}(\bH\bX)$ will also depend on the spectral and structural properties
of the matrix $\bH$.

\begin{proposition}
With the previous notation, assume that $y^*$ is a solution to (\ref{eqn:f_nonlin})
in $[\alpha,\beta]$. Then $\bX \equiv \bM + g(y^*) \bN$ is a
solution to (\ref{eqn:main}) with $f=g\circ h$. If $y^*$ is unique, then 
$\bX$ is also the unique solution to (\ref{eqn:main}).
\end{proposition}

\begin{proof}
Let $\bX\equiv\bM + g(y^*) \bN$. Applying the linear function $h$ to both sides we obtain
$h(\bX)=h(\bM) + g(y^*) h(\bN)$. We recall that
$h(\bM) + g(y^*) h(\bN) =  y^*$, therefore it must be that $h(\bX)=y^*$, that is,
$\bX=\bM + f(\bX) \bN$, which is equivalent to (\ref{eqn:main}).
\end{proof}

The quantities $\gamma_1, \gamma_2$ play a crucial role in the existence of (at least) one
solution to (\ref{eqn:f_nonlin}). In turn, these scalars depend on the eigenvalues of
the two Sylvester solutions, and thus on $\bA, \bB, \bC$ and $\bD$.
We abstain from exploring all possible cases of the nonlinear scalar problem, as our focus is 
on the difficulties stemming from the matrix setting. Below we give a sample of
theoretical and computational considerations that can be of help in solving the final problem,
keeping in mind that several other strategies could be used.

To explore the influence of the data on the nonlinear scalar equation, we 
assume $\gamma_2\ne 0$ and rewrite (\ref{eqn:f_nonlin}) as
\begin{equation}\label{eqn:graph}
g(y) =  -\frac{\gamma_1}{\gamma_2} + \frac{1}{\gamma_2} y ,
\end{equation}
and set $g_1(y)\equiv-\frac{\gamma_1}{\gamma_2} + \frac{1}{\gamma_2} y$, where
the function $g_1$ is linear and defined on the whole real line.
Hence, $y^*$ is a solution to (\ref{eqn:f_nonlin})
in $[\alpha,\beta]$ if and only if the two functions $g$ and $g_1$ intersect (at $y^*$). 
For simplicity, let us assume that $[\alpha,\beta]\equiv \RR$. If for instance 
$g$ (resp. $g_1$) is monotonically decreasing (resp. increasing) in $\RR$, then 
$y^*$ exists and is unique. This behavior depends on the choice of $g$, but also
on the sign of $\gamma_1$ and $\gamma_2$, which in turn depends
on the properties of the matrices $\bM, \bN$.
Examining all possible combinations of these properties would be cumbersome.
We provide here a typical setting.

\begin{proposition}
Assume that $\bM$ ($\bN$) is symmetric and positive (negative)
definite, and $g(y)\ge 0$ for any $y\ge 0$,
$g$ at least $C^2$ and monotonically decreasing. Then the Newton iteration $\{y_k\}$ 
applied to $F(y)=0$ with $F(y)=\gamma_1 + g(y)\gamma_2 - y$
will converge for any $y_0\ge 0$.
\end{proposition}

\begin{proof}
Note that the hypothesis on $\bM, \bN$ implies that $\gamma_1, \gamma_2$ are
both positive real values. Moreover, the hypotheses on $g$ also imply
that $F$ is at least $C^2$,
$F'(y)<0$ and $F''(y)>0$ for all positive $y$, so that $F$ is convex in $[0,+\infty)$.
Moreover,  $\lim_{y\to +\infty} F(y) = -\infty$. Since $F(0)>0$, a zero $y^*$ must
exist. 
The tangent passing through $y=0$ encounters the first coordinate axis
at $y_1= -F(0)/F'(0) >0$. Convexity ensures that $y_1<y^*$.
The tangent passing through $y=b$ for some $b>y^*$ encounters the first coordinate axis
at $y_1= b-F(b)/F'(b) = ( b g'(b)\gamma_2 -\gamma_1 - g(b) \gamma_2) /F'(b) > 0$
for all $b>0$. A known theorem ensures that the Newton iteration converges
in any interval $[0,b]$ with $y^* \in [0,b]$.
\end{proof}

\vskip 0.05in
\begin{example}
Let $g(t) = \exp(-t)$, so that $f(\bX) = \exp(-{\rm tr}(\bX))$.
Then (\ref{eqn:phi}) becomes $\gamma_1 - e^{-y}\gamma_2 - y = 0$, for $y \in  \RR$.
\end{example}
\vskip 0.05in

As an alternative to the Newton method, one can resort once again to a fixed point iteration. 
A natural choice, but not necessarily the best one, is given by
$$
y^{(k+1)} = \gamma_1+ g(y^{(k)})\gamma_2 \equiv \Phi(y^{(k)}).
$$
If a zero $y^*$ exists such that $|\Phi'(y^*)|<1$ then Ostrowski's theorem ensures that
there exists an open interval centered in $y^*$ such that this iteration will converge for any $y^{(0)}$
taken in this interval. Hence, the condition is that $|g'(y^*)\gamma_2| <1$.

As an example, let use take $g(y) = \ln(y)$, for $y>0$, so that $f(\bX)=\ln({\rm trace}(\bX))$. Then
$g'(y)= 1/y$ and $|\Phi'(y^*)|<1$ as long as $y^*>\gamma_2$. The existence of $y^*$ depends on
whether the curves $g(y)$ and $g_1(y)$ intersect, and as said around (\ref{eqn:graph}), this depends on
the mutual values of $\gamma_1, \gamma_2$.

\section{Considerations on the large scale case}
Problem (\ref{eqn:main}) becomes computationally very challenging if the given matrices have
large dimensions. Let $\widetilde \bM, \widetilde \bN$ be the approximations to the solutions 
$\bM$ and $\bN$ respectively, of the Sylvester equations.
If $f$ is linear, say $f(\bX)={\rm trace}(\bX)$, then 
from Proposition~\ref{prop:f(X)}  an approximate solution is obtained as
$$
\widetilde \bX \equiv \widetilde \bM +\sigma \widetilde \bN, \quad
\sigma = \frac{f(\widetilde\bM)}{1-f(\widetilde\bN)},
$$
with a clear dependence of the error $\bX - \widetilde \bX$ on the error committed
in approximating $\widetilde \bM, \widetilde \bN$.

For the approximation of $\bM, \bN$ different methods can be considered, especially in case the
right-hand sides $\bD$ and $\bC$ have low rank \cite{Simoncini.survey.16};
see also \cite{Palitta.Simoncini.18b} for the sparse setting.
Structural or sparsity properties are in fact a crucial hypothesis to be able to store 
$\widetilde \bM, \widetilde \bN$ and thus $\widetilde \bX$ in a memory saving, factored format.
Evaluating the trace can also profit from a factored form.
If projection methods are used to determine $\widetilde \bM, \widetilde \bN$ \cite{Simoncini.survey.16}, then 
the same type of projection strategy could be applied directly to (\ref{eqn:main}), so that the residual can be monitored
explicitly.
For the approximation $\widetilde \bX$ the associated residual is
$$
\bR = \bA \widetilde \bX + \widetilde \bX \bB + f(\widetilde \bX) \bC - \bD,
$$
which yields the following relation with the error matrix $\bE\equiv\widetilde \bX - \bX^\star$,
$$
\bR = \bA \bE + \bE \bB + (f(\widetilde \bX) - f(\bX^\star)) \bC.
$$
If $f$ is linear, then $f(\widetilde \bX) - f(\bX^\star) = f(\bE)$, hence it follows
$$
\bE = {\cal L}^{-1} (\bR) + f(\bE)\bN,
$$
which is the natural (linear) generalization of the known expression for the error matrix
in terms of the residual in linear algebraic equations.
%

Dealing with a nonlinear-linear $f$ is similar to the linear case, since the matrix $h(\bX)$ is a scalar,
after which the nonlinear function acts as in section~\ref{sec:nonlin_lin}.
The linear-nonlinear case analyzed in section~\ref{sec:lin_nonlin}
with large matrices is far more complicated. 
Assuming that the problem to be solved can again be written as
$\widetilde \bX = \widetilde \bM +f(\widetilde\bX) \widetilde \bN$,
a fixed-point iteration could be considered, possibly taking into account
memory saving representations of $\widetilde \bX, \widetilde \bM$ and $\widetilde \bN$, that is
$$
\widetilde \bX^{(k+1)} \equiv \widetilde \bM +f(\widetilde\bX^{(k)}) \widetilde \bN .
$$
However, how to {\it approximate} $f(\widetilde\bX^{(k)})$ remains complicated.
Consider for instance $f(\widetilde\bX)= {\rm trace}(\psi(\widetilde\bX))$.
The approximation of this function is a problem in its own, and different,
mostly iterative, approaches have been devised. This will give rise to an inner-outer
procedure for the fixed point scheme above.
Now popular choices for approximating 
${\rm trace}(\psi(\widetilde\bX))$ include randomized, Monte-Carlo and probing methods, which replace
the trace computation with the product
$z_k^T\psi(\bX) z_k$ for a selection of vectors $\{z_k\}$; see,
 e.g., \cite{cortinovis2020randomized},\cite{UbaruChenSaad.17} and their references.
Since in general we cannot expect high accuracy in this computation at each iteration,
the quality of the outer iteration may be considerably affected. A detailed analysis and
experimental study of these approaches is left for future research.

\section{Conclusions}
We have analyzed a new class of quasi-linear matrix equations, devising solutions in
closed form for the linear case. In the quasi-linear framework, we have proposed numerical
methods and theoretically studied their convergence under hypotheses that are satisfied for a wide
class of problem data. The large scale problem remains particularly challenging,
especially when involving the computation of matrix functions, for which further
work is required.

%

\bibliographystyle{siam}
\bibliography{/home/valeria/Bibl/Biblioteca,./marghe}

\end{document}